\documentclass[12pt]{article}
\usepackage{amsmath}
\usepackage{amsthm}
\usepackage{amssymb}
\usepackage{mathptmx}
\usepackage{graphicx}
\usepackage{esint}
\usepackage{algpseudocode}
\usepackage{url}
\usepackage[a4paper]{geometry}
\usepackage[active]{srcltx}
\usepackage{subcaption}

\newtheorem{Theorem}{Theorem}
\newtheorem{Proposition}{Proposition}
\newtheorem{rem}{Remark}
\newenvironment{Remark}{\rem\rm}{\endrem}
\newtheorem{Lemma}{Lemma}
\newtheorem{Definition}{Definition}

\newtheorem{assumption}{\bf Asumption}
\normalsize
\newcommand{\bd}{\begin{displaymath}}
\newcommand{\ed}{\end{displaymath}}
\newcommand{\be}{\begin{equation}}
\newcommand{\ee}{\end{equation}}
\newcommand{\bea}{\begin{eqnarray}}
\newcommand{\eea}{\end{eqnarray}}
\newcommand{\bda}{\begin{eqnarray*}}
\newcommand{\eda}{\end{eqnarray*}}
\newcommand{\ba}{\begin{array}}
\newcommand{\ea}{\end{array}}

\newcommand{\R}{\mathbb{R}}%
\DeclareMathOperator*\Graph{Graph}%
\DeclareMathOperator*\Zeros{Zeros}%

\begin{document}
	
\title{A Relaxed Inertial Forward-Backward-Forward Algorithm for Solving Monotone Inclusions with Application to GANs} 	
\author{Radu Ioan Bo{\c t}\footnote{Faculty of Mathematics, University of Vienna, Oskar-Morgenstern-Platz 1, 1090 Vienna, Austria, email: radu.bot@univie.ac.at. Research partially supported by the Austrian Science Fund (FWF), project number I 2419-N32.} \quad Michael Sedlmayer\footnote{Research Platform Data Science @ Uni Vienna, University of Vienna, W{\"a}hringer Stra{\ss}e 29, 1090 Vienna, Austria, e-mail: michael.sedlmayer@univie.ac.at.} \quad Phan Tu Vuong\footnote{Mathematical Sciences, University of Southampton, Highfield Southampton SO17 1BJ, United Kingdom, e-mail: T.V.Phan@soton.ac.uk.}} 
	\date{\today}
	
\maketitle

{\bf Abstract.} We introduce a relaxed inertial forward-backward-forward (RIFBF) splitting algorithm for approaching the set of zeros of the sum of a maximally monotone operator and a single-valued monotone and Lipschitz continuous operator. This work aims to extend Tseng's forward-backward-forward method by both using inertial effects as well as relaxation parameters. We formulate first a second order dynamical system which approaches the solution set of the monotone inclusion problem to be solved and provide an asymptotic analysis for its trajectories. We provide for RIFBF, which follows by explicit time discretization, a convergence analysis in the general monotone case as well as when applied to the solving of pseudo-monotone variational inequalities. We illustrate the proposed method by applications to a bilinear saddle point problem, in the context of which we also emphasize the interplay between the inertial and the relaxation parameters, and to the training of Generative Adversarial Networks (GANs).\vspace{1ex}

{\bf Key words.} forward-backward-forward algorithm, inertial effects, relaxation parameters, continuous time approach, application to GANs\vspace{1ex}

{\bf AMS subject classifications.} 47J20, 90C25, 90C30, 90C52

\section{Introduction} 
\label{sec:intro}
Let $H$ be a real Hilbert space endowed with inner product $\left\langle \cdot, \cdot \right\rangle$ and corresponding norm $\|\cdot\|$,
$A: H \to 2^H$ a maximally monotone set-valued operator, and $B: H \to H$ a monotone and Lipschitz continuous  operator with Lipschitz constant $L >0$, which means that
$$\|Bx-By\| \leq L\|x-y\| \quad \forall x,y \in H.$$
We recall that the operator $A:H\to 2^H$ is called monotone if
\begin{equation*}
\langle u-v,x-y \rangle \geq 0  \quad \forall (x,u), (y,v) \in \Graph(A),
\end{equation*}	
where $\Graph(A) = \{(x,y) \in H \times H : u \in Ax \}$ denotes its graph. The operator $A$ is called maximally monotone if it is monotone and its graph is not properly included in the graph of another monotone operator.

We assume that  $\Zeros (A+B) := \{x \in H : 0 \in Ax + Bx\} \neq \emptyset$, and are interested in solving the following inclusion problem: \vspace{1ex}

Find $x^*\in H$ such that
\begin{equation}\label{sMI}
0 \in Ax^*+Bx^*.
\end{equation}
An important special case is when $A=N_C$, the normal cone of a  nonempty closed convex subset $C$ of $H$. Then \eqref{sMI} reduces to a variational inequality (VI):\vspace{1ex}

Find $x^*\in C$ such that
\begin{equation} \label{VIP}
\left\langle Bx^*, x-x^*\right\rangle  \geq 0 \quad \forall x \in C.
\end{equation}
Solution methods for solving \eqref{sMI}, when the operator $B$ is cocoercive, have been developed intensively in the last decades. Recall that the operator $B : H \to H$ is cocoercive if there is a constant $L > 0$ such that
$$
L \left\langle Bx-By, x-y \right\rangle \ge \|Bx-By\|^2 \quad \forall x,y \in H. 
$$
Notice that every cocoercive operator is Lipschitz continuous and that the gradient of a convex and Fr\'echet differentiable function is a cocoercive operator if and only if it is Lipschitz continuous (\cite{BauschkeCombettes}). 

The simplest method for solving \eqref{sMI}, when $B$ is cocoercive, is the forward-backward (FB) method, which generates an iterative sequence $(x_k)_{k \geq 0}$ via
\begin{equation} \label{FB}
(\forall k \geq 0) \quad x_{k+1}:=J_{\lambda A} (I-\lambda B) x_k,
\end{equation}
where $x_0 \in H$ is the starting point and $J_{\lambda A} :=(I+\lambda A)^{-1} : H \to H$ is the resolvent of the operator $A$. Here, $I$ denotes the identity operator of $H$ and $\lambda$ is a positive stepsize chosen in $(0, \frac{2}{L})$. The resolvent of a maximally monotone operator is a single-valued and cocoercive operator with constant $L=1$. The iterative scheme \eqref{FB} results by time-discretizing with time stepsize equal $1$ the first order dynamical system
\begin{eqnarray*}
\begin{cases}
\begin{aligned}
& \dot{x}(t) + x(t) =  J_{\lambda A} \left( I-\lambda B\right) x(t), \\
&x(0)=x_0.
\end{aligned} 
\end{cases}
\end{eqnarray*}
For more about dynamical systems of implicit type associated to monotone inclusions and convex optimization problems, see  \cite{AA15, Antipin,Bolte03, BC_JMAA}.

Recently, the following second order dynamical system associated with the monotone inclusion problem \eqref{sMI},  when  $B$ is cocoercive,
\begin{eqnarray*}
\begin{cases}
\begin{aligned}
&\ddot{x}(t)+ \gamma(t) \dot{x}(t)+ \tau(t) \left[ x(t) -J_{\lambda A}(I - \lambda B)x(t) \right] =0,\\
&x(0)=x_0, \quad \dot{x}(0)=v_0,
\end{aligned} 
\end{cases}
\end{eqnarray*}
where $\gamma, \tau : [0,+\infty) \rightarrow [0,+\infty)$, was proposed and studied in \cite{BC_SICON} (see also \cite{BC_JMAA, Antipin, Alvarez00}). Explicit time discrerization of this second order dynamical system gives rise to so-called relaxed inertial forward-backward algorithms, which combine inertial effects and relaxation parameters. 

In the last years, Attouch and Cabot have promoted in a series of papers relaxed inertial algorithms for monotone inclusions and convex optimization problems, as they combine the advantages of both inertial effects and relaxation techniques. More precisely, they addressed the relaxed inertial proximal method (RIPA) in \cite{AC_Prox,AC_ProxRate} and the relaxed inertial forward-backward method (RIFB) in \cite{AC_FB}. A relaxed inertial Douglas-Rachford algorithm for monotone inclusions has been proposed in \cite{BCH_iDR}. Iutzeler and Hendrickx investigated in \cite{IutzelerHendrickx} the influence inertial effects and relaxation techniques have on the numerical performances of optimization algorithms. The interplay between relaxation and inertial parameters for relative-error inexact under-relaxed algorithms has been addressed in \cite{MA1, MA2}.

Relaxation techniques are essential ingredients in the formulation of algorithms for monotone inclusions, as they provide more flexibility to the iterative schemes (see \cite{BauschkeCombettes, EcksteinBertsekas}). Inertial effects have been introduced in order to accelerate the convergence of the numerical methods. This technique traces back to the pioneering work of Polyak \cite{Polyak64}, who introduced the heavy ball method in order to speed up the convergence behaviour of the gradient algorithm and allow the detection of different critical points. This idea was employed and refined by Nesterov (see \cite{Nesterov83}) and by Alvarez and Attouch (see \cite{Alvarez00, AA01}) in the context of solving smooth convex minimization problems and monotone inclusions/nonsmooth convex minimization problems, respectively. In the last decade, an extensive literature has been devoted to inertial algorithms.

In this paper we will focus on the solving of the monotone inclusion \eqref{sMI} in the case when $B$ is merely monotone and Lipschitz continuous. To this end we will formulate a relaxed inertial forward-backward-forward (RIFBF) algorithm, which we obtain through the time discretization of a second order dynamical system approaching the solution set of \eqref{sMI}. The forward-backward-forward (FBF) method was proposed by Tseng in \cite{Tseng2001} and it generates an iterative sequence $(x_k)_{k \geq 0}$ via
\begin{equation*} 
		(\forall k \geq 0) \quad \begin{cases}
			y_k = J_{\lambda_k A} (I-\lambda_k B) x_k,\\
			x_{k+1} =  y_k-\lambda_k(By_k-Bx_k),
		\end{cases}
	\end{equation*}
where $x_0 \in H$ is the starting point. The sequence $(x_k)_{k \geq 0}$ converges weakly to a solution of \eqref{sMI} if the sequence of stepsizes $(\lambda_k)_{k \geq 0}$ is chosen in the interval $\left(  0,\frac{1}{L} \right)$, where $L > 0$ is the Lipschitz constant of $B$. An inertial version of FBF was proposed in \cite{BC_iFBF}. 

Recently, a forward-backward algorithm for solving \eqref{sMI}, when $B$ is monotone and Lipschitz continuous, was proposed by Malitsky and Tam in \cite{MaTa}. This method requires in every iteration only one forward step instead of two, however, the sequence of stepsizes has to be chosen constant in the interval $\left(0, \frac{1}{2L}\right)$, which slows the algorithm down in comparison to FBF. A popular algorithm used to solve the variational inequality \eqref{VIP}, when $B$ is monotone and Lipschitz continuous, is Korpelevich extrargadient's method (see \cite{Korpelevich}). The stepsizes are to be chosen in the interval $\left(0, \frac{1}{2L}\right)$, however, this method requires two projection steps on $C$ and two forward steps.

The main motivation for the investigation of monotone inclusions and variational inequalities governed by monotone and Lipschitz continuous operators is represented by minimax convex-concave problems. It is well-known that determining  primal-dual pairs of optimal solutions of convex optimization problems means actually solving minimax convex-concave problems (see \cite{BauschkeCombettes}). Minimax problems arise traditionally in game theory and, more recently, in the training of Generative Adversarial Networks (GANs), as we will see in Section \ref{sec:Numerical}.

In the next section we will approach the solution set of \eqref{sMI} from a continuous perspective by means of the trajectories generated by a second order dynamical system of FBF type.
We will prove an existence and uniqueness result for the generated trajectories and  provide a general setting in which these converge to a zero of $A+B$ as time goes to infinity. In addition, we will show that explicit time discretization of the dynamical system gives rise to an algorithm of forward-backward-forward type with inertial and relaxation parameters (RIFBF).

In Section~\ref{sec:iFBF}) we will discuss the convergence of (RIFBF) and investigate the  interplay between the inertial and the relaxation parameters.  It is of certain relevance to notice that both the standard FBF method, the algorithm in \cite{MaTa} and the extragradient method require to know the Lipschitz constant of $B$, which is not always available. This can be avoided by performing a line-search procedure, which usually leads to additional computation costs. On the contrary, we will use an adaptive stepsize rule which does not require knowledge of the Lipschitz constant of $B$. We will also comment on the convergence of (RIFBF) when applied to the solving of the variational inequality \eqref{VIP} in the case when the operator $B$ is pseudo-monotone but not necessarily monotone. Pseudo-monotone operators appear in the consumer theory of mathematical economics (\cite{HadjisavvasSchaibleWong}) and as gradients of pseudo-convex functions (\cite{CottleFerland}), such as ratios of convex and concave functions in fractional programming (\cite{BL06}).

Concluding, we treat two different numerical experiments supporting our theoretical results in Section~\ref{sec:Numerical}. On the one hand we deal with a bilinear saddle point problem which can be understood as a two-player zero-sum constrained game. In this context, we emphasize the interplay between the inertial and the relaxation parameters. On the other hand we employ variants of (RIFBF) for training generative adversarial networks (GANs), which is a class of machine learning systems where two opposing artificial neural networks compete in a zero-sum game. GANs have achieved outstanding results for producing photorealistic pictures and are typically known to be difficult to optimize. We show that our method outperform ``Extra Adam'', a GAN training approach inspired by the extra-gradient algorithm, which recently achieved state-of-the-art results (see \cite{VIP-GAN}).

\section{A second order dynamical system of FBF type} 
\label{sec:dynamicFBF}
In this section we will focus on the study of the  dynamical system
\begin{eqnarray}\label{DS}
\begin{cases}
\begin{aligned}
&y(t)=J_{\lambda A}(I - \lambda B)x(t),\\
&\ddot{x}(t)+ \gamma(t) \dot{x}(t)+ \tau(t) \left[ x(t) -y(t) - \lambda\left(Bx(t)-By(t)\right) \right] =0,\\
&x(0)=x_0, \quad \dot{x}(0)=v_0,
\end{aligned} 
\end{cases}
\end{eqnarray}
where $\gamma, \tau: [0, +\infty) \to [0, +\infty)$ are Lebesgue measurable function,  $0 < \lambda < \frac{1}{L}$ and $x_0, v_0 \in H$, in connection with the monotone inclusion problem \eqref{sMI}.

We define $M : H \to H$ by 
\begin{equation}\label{DST}
Mx=x - J_{\lambda A}(I - \lambda B)x - \lambda \left[Bx-B\circ J_{\lambda A}(I - \lambda B)x\right].
\end{equation}
Then \eqref{DS} can be equivalently written as 
\begin{eqnarray}\label{DS1}
\begin{cases}
\begin{aligned}
&\ddot{x}(t)+ \gamma(t) \dot{x}(t)+ \tau(t) M x(t) =0,\\
&x(0)=x_0,\quad \dot{x}(0)=v_0.
\end{aligned} 
\end{cases}
\end{eqnarray}

The following result collects some properties of $M$. 

\begin{Proposition} \label{Tproperties}
Let $M$ be defined as in~\eqref{DST}. Then the following statements hold:
\begin{itemize}
\item[(i)] $\Zeros(M) = \Zeros (A+B)$;
\item[(ii)] $M$ is Lipschitz continuous;
\item[(iii)] There exists a positive constant $\kappa>0$ such that for all $x^* \in  \Zeros(M)$ and all $x \in H$ it holds
$$
\left\langle Mx, x-x^*\right\rangle \ge \kappa \|Mx\|^2.
$$ 
\end{itemize}
\end{Proposition} 
\begin{proof}
	
(i) For $x \in H$ we set $y:=J_{\lambda A}(I - \lambda B)x$, thus $Mx = x - y - \lambda \left(Bx-By\right)$. Using the Lipschitz continuity of $B$ we have
\begin{equation} \label{kep}
	(1-\lambda L) \|x - y\| \le  \|Mx\| = \|x - y - \lambda \left(Bx-By\right)\|=(1+\lambda L)\|x - y\|.
\end{equation}
	Therefore, $x \in \Zeros(M)$ if and only if $x=y = J_{\lambda A}(I-\lambda B) x$, which is further equivalent to $ x \in \Zeros(A+B)$.
	
(ii) Let $x, x' \in H$ and $y:=J_{\lambda A}(I - \lambda B)x$, and $y':=J_{\lambda A}(I - \lambda B)x'$. The Lipschitz continuity of $B$ yields
\begin{eqnarray*}
\|Mx-Mx'\| &=& \|x - y - \lambda \left(Bx-By\right) - x' + y' + \lambda \left(Bx'-By'\right)\|\\
&\le& (1+\lambda L)\left(\|x-x'\|+\|y-y'\| \right). 
\end{eqnarray*}
In addition, by the non-expansiveness (Lipschitz continuity with Lipschitz constant $1$) of $J_{\lambda A}$ and again by the Lipschitz continuity of $B$ we obtain 
\begin{eqnarray*}
\|y-y'\| &=& \|J_{\lambda A}(I - \lambda B) x -J_{\lambda A}(I - \lambda B) x'\|\\
&\le& \|(I - \lambda B) x -(I - \lambda B) x'\| \le (1+\lambda L)\|x-x'\|. 
\end{eqnarray*}
Therefore, 
$$
\|Mx-Mx'\|  \le (1+\lambda L)(2+\lambda L)\|x-x'\|,
$$
which shows that $M$ is Lipschitz continuous with Lipschitz constant $(1+\lambda L)(2+\lambda L)>0$.

(iii) Let $x^* \in H$ be such that $0 \in (A+B)x^*$ and $x \in H$. We denote $y := J_{\lambda A}(I-\lambda B) x$ and can write $(I-\lambda B) x \in (I +\lambda A) y$
or, equivalently,
\begin{align}\label{eV2}
\frac{1}{\lambda} \left( x-y\right) -\left( Bx-By\right) \in (A+B)y. 
\end{align}
Using the monotonicity of $A+B$ we obtain
$$
\left\langle \frac{1}{\lambda} \left( x-y\right) -\left( Bx-By\right), y-x^* \right\rangle \ge 0,
$$ 
which is equivalent to
$$
\left\langle x-y- \lambda \left( Bx-By\right), x-x^*\right\rangle \ge \left\langle x-y- \lambda \left( Bx-By\right), x-y\right\rangle.
$$
This means that
\begin{eqnarray*}
\left\langle Mx, x-x^*\right\rangle 
&\ge &\left\langle x-y- \lambda \left( Bx-By\right), x-y\right\rangle = \|x-y\|^2 -\lambda \left\langle Bx-By, x-y \right\rangle \\
&\ge& \|x-y\|^2 -\lambda \|Bx-By\| \| x-y \| \ge (1-\lambda L) \|x-y\|^2 \\
&\ge& \frac{1-\lambda L}{(1+\lambda L)^2} \|Mx\|^2,
\end{eqnarray*}
where the last inequality follows from \eqref{kep}. 
Therefore, (iii) holds with $\kappa := \frac{1-\lambda L}{(1+\lambda L)^2}>0$.
\end{proof}

The following definition makes explicit which kind of solutions of the dynamical system \eqref{DS} we are looking for. We recall that a function $x: [0,b] \to H$ (where $b>0$) is said to be absolutely continuous if there exists an integrable function $y:[0,b] \to H$ such that
	$$
	x(t) = x(0) + \int_{0}^{t} y(s) ds \quad \forall t \in[0,b].
	$$
This is nothing else than $x$ is continuous and its distributional derivative $\dot{x}$ is Lebesgue integrable on $[0,b]$.  

\begin{Definition}
We say that $x: [0,+\infty) \to H$ is a strong global solution of \eqref{DS} if the following properties are satisfied:
\begin{itemize}
	\item[(i)] $x$,  $\dot{x}: [0,+\infty) \to H$ are locally absolutely continuous, in other words, absolutely continuous on each interval $[0,b]$ for $0<b<+\infty$;
	\item[(ii)] $\ddot{x}(t)+ \gamma(t) \dot{x}(t)+ \tau(t) M x(t) =0 $ for almost every $t \in [0, +\infty)$;
	\item[(iii)] $x(0) = x_0$ and $\dot{x}(0)=v_0$.
\end{itemize}
\end{Definition}

Since $M : H \rightarrow H$ is Lipschitz continuous, the existence and uniqueness of the trajectory of \eqref{DS} follows from the Cauchy-Lipschitz Theorem for absolutely continuous trajectories.

\begin{Theorem} (see \cite[Theorem 4]{BC_SICON})
Let $\gamma, \tau: [0, +\infty) \to  [0, +\infty)$ be Lebesgue measurable functions such that $\gamma, \tau \in L^1_{loc}([0, +\infty))$ (that is $\gamma, \tau \in L^1([0, b])$ for all $0<b<+\infty$). Then for each $x_0,v_0 \in H$ there exists a unique strong global solution of the dynamical system \eqref{DS}.
\end{Theorem}

We will prove the convergence of the trajectories of \eqref{DS} in a setting which requires the damping function $\gamma$ and the relaxation function $\tau$ to fulfil the assumptions below.  We refer to \cite{BC_SICON} for examples of functions which fulfil this assumption and want also to emphasize that when the two functions are constant we recover the conditions from \cite{Attouch_Mainge}.

\begin{assumption}\label{AssumpPara}
$\gamma, \tau: [0, +\infty) \to  [0, +\infty)$ are locally absolutely continuous and there exists $\theta > 0$ such that for almost every $t \in [0, +\infty)$ it holds
$$
\dot{\gamma}(t) \le 0 \le \dot{\tau} (t) \quad \text{and} \quad \frac{\gamma^2 (t)}{\tau (t)} \ge \frac{1+\theta}{\kappa}.
$$
\end{assumption}

The result which states the convergence of the trajectories is adapted from \cite[Theorem 8]{BC_SICON}. Though, it cannot be obtained as a direct consequence of it, since the operator $M$ is not cocoercive as it is required to be in \cite[Theorem 8]{BC_SICON}. However, as seen in Proposition \ref{Tproperties} (iii), $M$ has a property, sometimes called ``coercive with respect to its set of zeros'', which is by far weaker than coercivity, but strong enough in order to allow us to partially apply the techniques used to prove \cite[Theorem 8]{BC_SICON}.

\begin{Theorem}
Let $\gamma, \tau: [0, +\infty) \to  [0, +\infty)$ be functions satisfying Assumption \ref{AssumpPara} and $x_0, v_0 \in H$. Let $x: [0, +\infty) \to  H$ be the unique strong global solution of \eqref{DS}. Then the following statements are true:
\begin{itemize}
	\item[(i)] the trajectory $x$ is bounded and $\dot{x}, \ddot{x}, Mx \in L^2 \left([0,+\infty); H \right)$;
	\item[(ii)] $\lim_{t \to +\infty } \dot{x}(t) = \lim_{t \to +\infty } \ddot{x}(t) =\lim_{t \to +\infty } Mx(t) = \lim_{t \to +\infty } \left[ x(t)-y(t)\right]=0$;
	\item[(iii)] $x(t)$ converges weakly to an element in $\Zeros (A+B)$ as $t \to +\infty$.
\end{itemize}
\end{Theorem}
\begin{proof}
Take an arbitrary $x^* \in \Zeros (A+B) = \Zeros(M)$ and define for all $t \in [0, +\infty)$ the Lyapunov function $h(t)=\frac{1}{2}\|x(t)-x^*\|^2$. For almost every $t \in [0, +\infty)$ we have
$$
\dot{h}(t) = \left\langle x(t)-x^*, \dot{x}(t)\right\rangle \ \mbox{and} \ \ddot{h}(t) = \|\dot{x}(t)\|^2+\left\langle x(t)-x^*, \ddot{x}(t)\right\rangle.
$$
Taking into account \eqref{DS1} we obtain for almost every $t \in [0, +\infty)$ that
$$
\ddot{h}(t) +\gamma (t) \dot{h}(t) +\tau(t) \left\langle x(t)-x^*, Mx(t) \right\rangle  = \|\dot{x}(t)\|^2, 
$$
which, together with  Proposition \ref{Tproperties} (iii), implies
$$
\ddot{h}(t) +\gamma (t) \dot{h}(t) +\kappa \tau(t) \|Mx(t)\|^2  \le \|\dot{x}(t)\|^2.$$

From this point, we can proceed as in the proof of \cite[Theorem 8]{BC_SICON} and, consequently, obtain  the statements in (i) and (ii) and the fact that the limit $\lim_{t \to +\infty} \|x(t)-x^*\| \in \R$ exists,
which is the first assumption in the continuous version of the Opial Lemma (see, for instance, \cite[Lemma 7]{BC_SICON}). In order to show that the second assumption of the Opial Lemma is fulfilled, which means actually proving that every weak sequential cluster point of the trajectory $x$ is a zero of $M$, one cannot use the arguments in \cite[Theorem 8]{BC_SICON}, since $M$ is not maximal monotone. We have to use, instead, different arguments relying on the maximal monotonicity of $A+B$.

Indeed, let $\bar{x}$ be a weak sequential cluster point of $x$, which means that there exists a sequence $t_k \to +\infty$ such that $(x(t_k))_{k \geq 0}$ converges weakly to $\bar{x}$ as $k \to +\infty$. Since, according to (ii), $\lim_{t \to +\infty } Mx(t) = \lim_{t \to +\infty } \left[ x(t)-y(t)\right]=0$, we conclude that $\{y(t_k)\}_{k \geq 0}$ also converges weakly to $\bar{x}$. According to \eqref{eV2} we have
\begin{equation} \label{eV3}
\frac{1}{\lambda} \left( x(t_k)-y(t_k)\right) -\left( Bx(t_k)-By(t_k)\right)  \in (A+B)y(t_k) \quad \forall k \geq 0.
\end{equation}
Since $B$ is Lipschitz continuous and $\lim_{t \to +\infty } \| x(t_k)-y(t_k) \|=0$, the left hand side of \eqref{eV3} converges strongly to $0$ as $k \to +\infty$. Since $A+B$ is maximal monotone, its graph is sequentially closed with respect to the weak-strong topology of the product space $H \times H$. Therefore, taking the limit as $t_k \to +\infty$ in \eqref{eV3} we obtain $\bar{x} \in \Zeros (A+B)$.

Thus, the continuous Opial Lemma implies that $x(t)$ converges weakly to an element in $\Zeros (A+B)$ as $t \to +\infty$.
\end{proof}

\begin{Remark}\label{expdis} {\bf (explicit discretization)} Explicit time discretization of \eqref{DS} with stepsize $h_k >0$, relaxation variable $\tau_k >0$, damping variable $\gamma_k >0$, and initial points $x_0$ and $x_1$ yields for all $k \geq 1$ the following iterative scheme:
\begin{equation} \label{eV13}
\frac{1}{h_k^2} \left( x_{k+1}-2 x_k +x_{k-1} \right) + \frac{ \gamma_k}{h_k} \left( x_{k}-x_{k-1} \right) 
	+ \tau_k Mz_k=0,
\end{equation}
where $z_k$ is an extrapolation of $x_k$ and $x_{k-1}$ that will be chosen later. The Lipschitz continuity of $M$ provides a certain flexibility to this choice. We can write \eqref{eV13} equivalently as
$$
(\forall k \geq 1) \quad x_{k+1} =x_k+(1-\gamma_k h_k) (x_k-x_{k-1}) - h_k^2 \tau_k Mz_k.
$$  
Setting $\alpha_k := 1-\gamma_k h_k$,  $\rho_k := h_k^2 \tau_k$ and choosing $z_k:=x_k+\alpha_k (x_{k}-x_{k-1})$ for all $k \geq 1$, we can write the above scheme as
\begin{eqnarray*} 
	(\forall k \geq 1) \ \begin{cases}
	z_k=x_k+\alpha_k(x_k-x_{k-1})\\
		y_k=J_{\lambda A} (I-\lambda B) z_k\\
		x_{k+1}=(1-\rho_k) z_k +\rho_k \left(y_k-\lambda (By_k-Bz_k) \right),		
	\end{cases}
\end{eqnarray*}
 which is a relaxed version of the FBF algorithm with inertial effects. 
\end{Remark}

\section{A relaxed inertial FBF algorithm} 
\label{sec:iFBF}
In this section we investigate the convergence of the relaxed inertial algorithm derived in the previous section via the time discretization of \eqref{DS}, however, we also assume that the stepsizes $(\lambda_{k})_{k \geq 1}$ are variable. More precisely, we will address the following algorithm
\begin{equation*}
	(RIFBF) \quad \quad \quad (\forall k \geq 1) \
	\begin{cases}
	z_k=x_k+\alpha_k(x_k-x_{k-1})\\
	y_k=J_{\lambda_k A} (I-\lambda_k B) z_k\\
	x_{k+1}=(1-\rho_k) z_k +\rho_k \left(y_k-\lambda_k(By_k-Bz_k) \right),	
	\end{cases}
\end{equation*}
where $x_0,x_1 \in H$ are starting points, $(\lambda_{k})_{k \geq 1}$ and $(\rho_{k})_{k \geq 1}$ are sequences of positive numbers, and $(\alpha_{1})_{k \geq 1}$ is a sequence of non-negative numbers. The following iterative schemes can be obtained as particular sequences of (RIFBF):

\begin{itemize}
	\item  $\rho_k=1$ for all $k \geq 1$: inertial Forward-Backward-Forward algorithm
	\begin{equation*}
		(IFBF) \quad \quad \quad (\forall k \geq 1) \
		\begin{cases}
			z_k=x_k+\alpha_k(x_k-x_{k-1})\\
			y_k=J_{\lambda_k A} (I-\lambda_k B) z_k\\
			x_{k+1}=y_k-\lambda_k(By_k-Bz_k)
		\end{cases}
	\end{equation*}
	\item  $\alpha_k=0$ for all $k \geq 1$: relaxed  Forward-Backward-Forward algorithm
	\begin{equation*}
		\hspace{2.5cm}(RFBF) \quad \quad \quad (\forall k \geq 1) \
		\begin{cases}
			y_k=J_{\lambda_k A} (I-\lambda_k B) x_k\\
			x_{k+1}=(1-\rho_k) x_k +\rho_k \left(y_k-\lambda_k(By_k-Bx_k) \right)		
		\end{cases}
	\end{equation*}
	\item  $\alpha_k=0$,  $\rho_k=1$ for all $k \geq 1$: Forward-Backward-Forward algorithm
	\begin{equation*}
		(FBF) \quad \quad \quad (\forall k \geq 1) \
		\begin{cases}
			y_k = J_{\lambda_k A} (I-\lambda_k B) x_k\\
			x_{k+1}=y_k-\lambda_k(By_k-Bx_k).
		\end{cases}
	\end{equation*}
\end{itemize}

{\bf Stepsize rules:} 

Depending on the availability of the Lipschitz constant $L$ of $B$ , we have two different options for the choice of the sequence of stepsizes $(\lambda_k)_{k \geq 1}$:
\begin{itemize}
	\item {\bf constant stepsize}: $\lambda_k := \lambda \in \left(0, \frac{1}{L}\right)$ for all $k \ge 1$;
	\item {\bf adaptive stepsize}: let $\mu \in (0,1)$ and $ \lambda_1 > 0 $. The stepsizes for $k \geq 1$ are adaptively updated as follows
	\begin{equation} \label{stepsize}
		\lambda_{k+1} :=
		\begin{cases}
			\min \left\lbrace 
		 		\lambda_{k}, \frac{\mu \Vert y_{k} - z_{k} \Vert}{ \Vert B y_{k} - B z_{k} \Vert}
		 	\right\rbrace, 
			& \text{ if } B y_{k} - B z_{k} \not= 0,\\
			\lambda_{k},
			& \text{ otherwise.}
		\end{cases}
	\end{equation}
\end{itemize}

If the Lipschitz constant $L$ of $B$ is known in advance, then a constant stepsize can be chosen. Otherwise, the adaptive stepsize rule \eqref{stepsize} is highly recommended. In the following, we provide the convergence analysis for the adaptive stepsize rule, as the constant stepsize rule can be obtained as a particular of it by setting  $\lambda_1:=\lambda$ and $\mu := \lambda L$.

\begin{Proposition} \label{Lem1} 
Let $\mu \in (0,1)$ and $ \lambda_0 > 0$. The sequence $\left(\lambda_k \right)_{k \geq 1}$ generated by \eqref{stepsize} is nonincreasing and 
	$$
	\lim_{k\to +\infty} \lambda_k=\lambda \geq 
	\min\left\lbrace \lambda_1,\dfrac{\mu}{L}\right\rbrace.
	$$
In addition,
	\begin{equation} \label{Lips}
	\|By_{k}-Bz_{k}\| \leq \frac{\mu}{\lambda_{k+1}} \|y_{k}-z_{k}\| \quad \forall k \geq 1.
	\end{equation}
	\end{Proposition}
\begin{proof} It is obvious from ~\eqref{stepsize} that  $\lambda_{k+1}\le \lambda_k$ for all $k\geq  1$. Since $B$ is Lipschitz continuous with Lipschitz constant $L$, it yields
	\begin{equation*}
	\frac{\mu \|y_{k}-z_{k}\|}{\|By_{k}-Bz_{k}\|} \geq \dfrac{\mu}{L}, \ \text{if} \ By_{k}-Bz_{k} \not= 0,
	\end{equation*}
	which together with \eqref{stepsize} yields
	\begin{equation*}
	\lambda_{k+1}\geq \min\left\lbrace \lambda_1,\dfrac{\mu}{L}\right\rbrace \quad \forall k \geq 1.
	\end{equation*}
Thus, there exists 
	$$
	\lambda:=\lim_{k\to +\infty} \lambda_k  \geq 
	\min\left\lbrace \lambda_1,\dfrac{\mu}{L}\right\rbrace.
	$$
	The inequality \eqref{Lips} follows directly from \eqref{stepsize}.
\end{proof}

\begin{Proposition} \label{Prop1}
Let $t_k:=y_k -  \lambda_k (By_k-Bz_k)$ and $\theta_k:=\frac{\lambda_k }{\lambda_{k+1}}$ for all $k \geq 1$. Then for all $x^* \in \Zeros(A+B)$ it holds
	\begin{equation}\label{MainIne}
	\|t_k-x^*\|^2\le \|z_k-x^*\|^2-\left(1-\mu ^2 \theta_k ^2\right) \|y_{k} - z_k\|^2 \quad \forall k \ge 1.
	\end{equation}
\end{Proposition}
\begin{proof}
Let $k \geq 1$. Using \eqref{Lips} we have that
\begin{eqnarray}\label{intermed}
	\|z_k-x^*\|^2 &=& \|z_k-y_k+y_k-t_k+t_k-x^*\|^2\nonumber\\ 
	&=& \|z_k-y_k\|^2+\|y_k-t_k\|^2+\|t_k-x^*\|^2\nonumber\\
	&& \quad + 2\left\langle z_k-y_k,y_k-x^*\right\rangle +2\left\langle y_k-t_k,t_k-x^*\right\rangle \nonumber\\
	&=& \|z_k-y_k\|^2-\|y_k-t_k\|^2+\|t_k-x^*\|^2+2\left\langle z_k-t_k,y_k-x^*\right\rangle\nonumber \\
	&=& \|z_k-y_k\|^2-\lambda_k ^2 \|By_k-Bz_k\|^2+\|t_k-x^*\|^2+2\left\langle z_k-t_k,y_k-x^*\right\rangle \nonumber\\
	&\geq & \|z_k-y_k\|^2-\frac{\mu ^2 \lambda_k ^2}{\lambda_{k+1} ^2}  \|y_k-z_k\|^2+\|t_k-x^*\|^2+2\left\langle z_k-t_k,y_k-x^*\right\rangle.
\end{eqnarray}
On the other hand, since
$$
(I+\lambda_k A) y_k \ni (I-\lambda_k B)z_k,
$$
we have
$$
z_k \in y_k + \lambda_k A y_k+\lambda_k Bz_k=y_k - \lambda_k (By_k-Bz_k)+\lambda_k(A+B)y_k=t_k+\lambda_k(A+B)y_k.
$$
Therefore, 
$$
\frac{1}{\lambda_k}(z_k - t_k )\in (A+B)y_k, 
$$
which, together with $0 \in (A+B)x^*$ and the monotonicity of $A+B$, implies
$$
\left\langle z_k-t_k,y_k-x^*\right\rangle \ge 0.
$$
Hence,
$$
	\|z_k-x^*\|^2 \ge  \|z_k-y_k\|^2-\frac{\mu ^2 \lambda_k ^2}{\lambda_{k+1} ^2} \|y_k-z_k\|^2+\|t_k-x^*\|^2,
$$
which is \eqref{MainIne}.
\end{proof}

The following result introduces a discrete Lyapunov function for which a decreasing property is established.

\begin{Proposition} \label{Lem2bis} Let $(\alpha_k)_{k \geq 1}$ be a non-decreasing sequence of non-negative numbers and $(\rho_k)_{k \geq 1}$ a sequence of positive numbers, let $x^* \in \Zeros(A+B)$ and for all $k \geq 1$ define
	\begin{align}\label{Hk}
	H_k :=\|x_k-x^*\|^2 - \alpha_k \|x_{k-1}-x^*\|^2 
	+  2\alpha_k \left(\alpha_k+ \frac{1-\alpha_k}{\rho_k (1+\mu \theta_k)} \right) \|x_k - x_{k-1}\|^2.
	\end{align}	
	Then there exists $ k_{0} \geq 1$ such that for all $ k \geq k_{0} $ it holds
	\begin{equation} \label{Descent1}
	H_{k+1}-H_k  \leq - \delta_k \|x_{k+1}-x_k\|^2,
	\end{equation}
	where 
	$$
	\delta_k:=  \left( 1-\alpha_k \right) \left(\frac{2}{\rho_k (1+\mu \theta_k)}-1 \right) 
	-  2\alpha_{k+1} \left(\alpha_{k+1}+ \frac{1-\alpha_{k+1}}{\rho_{k+1} (1+\mu \theta_{k+1})} \right) \quad \forall k \geq 1.
	$$
\end{Proposition}
\begin{proof}
From Proposition \ref{Prop1} we have for all $k \geq 1$
\begin{eqnarray} \label{F1bis}
	\|x_{k+1}-x^*\|^2 \nonumber
	&=& \|(1-\rho_k) z_k + \rho_k t_k-x^*\|^2 \\\nonumber 
	&=& \|(1-\rho_k) (z_k-x^*) + \rho_k (t_k-x^*)\|^2 \\\nonumber 
	&=& (1-\rho_k) \|z_k-x^*\|^2 + \rho_k \|t_k-x^*\|^2-\rho_k (1-\rho_k) \|t_k-z_k\|^2 \\\nonumber 
	&\leq & (1-\rho_k) \|z_k-x^*\|^2 + \rho_k \|z_k-x^*\|^2 \\\nonumber
	&& \ - \rho_k \left(1-\mu^2 \theta^2_k\right) \|y_{k} - z_k\|^2 -\frac{1-\rho_k}{\rho_k} \|x_{k+1}-z_k\|^2 \\ 
		&= & \|z_k-x^*\|^2 - \rho_k \left(1-\mu^2 \theta^2_k\right) \|y_{k} - z_k\|^2 -\frac{1-\rho_k}{\rho_k} \|x_{k+1}-z_k\|^2.
\end{eqnarray}
Using \eqref{Lips} we obtain for all $k \geq 1$
\begin{eqnarray*} \label{F10}
\frac{1}{\rho_k} \|x_{k+1}-z_k\| =\|t_k-z_k\| \nonumber
&\leq &  \|t_k-y_k\|+\|y_k-z_k\| =  \lambda_k\|By_k-Bz_k\|+\|y_k-z_k\| \\
&\leq &  \left( 1+\frac{\mu \lambda_k }{\lambda_{k+1}}\right) \|y_k-z_k\|
=\left( 1+ \mu \theta_k\right) \|y_k-z_k\|.
\end{eqnarray*}
Since $\lim_{k \to +\infty} \left(1-\mu^2 \theta_k^2\right) = 1-\mu^2 >0 $, there exists $ k_{0} \geq 1 $ such that 
$$
1-\mu^2 \theta_k^2 > \frac{1-\mu^2}{2} \quad \forall k \geq k_0.
$$
This means that for all $k \ge k_0$, we have
$$
\frac{1-\mu \theta_k}{\rho_k (1+\mu \theta_k)}\|x_{k+1}-z_k\|^2 \le \rho_k \left(1-\mu^2 \theta^2_k\right) \|y_k-z_k\|^2. 
$$
We obtain from \eqref{F1bis} that for all $k \geq k_0$
\begin{eqnarray}\label{VV3}
\|x_{k+1}-x^*\|^2 \nonumber
&\le & \|z_k-x^*\|^2 - \left(\frac{1-\mu \theta_k}{\rho_k (1+\mu \theta_k)}+\frac{1-\rho_k}{\rho_k} \right) \|x_{k+1}-z_k\|^2\\
&= & \|z_k-x^*\|^2 - \left(\frac{2}{\rho_k (1+\mu \theta_k)}-1 \right) \|x_{k+1}-z_k\|^2.
\end{eqnarray}
We will estimate the right-hand side of \eqref{VV3}. For all $k \geq 1$ we have
\begin{eqnarray} \label{F22}
\|z_{k}-x^*\|^2 \nonumber
&=& \|x_k+\alpha_k(x_k-x_{k-1})-x^*\|^2 \\\nonumber 
&=& \|(1+\alpha_k) (x_k-x^*) - \alpha_k (x_{k-1}-x^*)\|^2 \\
&=& (1+\alpha_k) \|x_k-x^*\|^2 - \alpha_k \|x_{k-1}-x^*\|^2 + \alpha_k(1+\alpha_k) \|x_k - x_{k-1}\|^2,
\end{eqnarray}
and
\begin{eqnarray} \label{F23}
\|x_{k+1}-z_k\|^2 \nonumber
&= &  \| (x_{k+1}-x_k)-\alpha_k(x_k-x_{k-1})\|^2 \\\nonumber 
&=&  \|x_{k+1}-x_k\|^2+\alpha_k^2 \|x_k-x_{k-1}\|^2 - 2 \alpha_k \left\langle x_{k+1}-x_k, x_k-x_{k-1} \right\rangle \\
&\ge & \left( 1-\alpha_k \right) \|x_{k+1}-x_k\|^2+\left( \alpha_k^2 - \alpha_k \right) \|x_k-x_{k-1}\|^2.
\end{eqnarray}
Combining \eqref{VV3} with \eqref{F22} and \eqref{F23}  and using that $(\alpha_k)_{k \geq 1}$ is non-decreasing, we obtain for all $k \geq k_0$
\begin{eqnarray} 
\nonumber &&\|x_{k+1}-x^*\|^2 -\alpha_{k+1} \|x_k-x^*\|^2 \\ \nonumber 
&\le &\|x_{k+1}-x^*\|^2 -\alpha_k \|x_k-x^*\|^2 \\ \nonumber
&\le& \|x_k-x^*\|^2 - \alpha_k \|x_{k-1}-x^*\|^2 + \alpha_k(1+\alpha_k) \|x_k - x_{k-1}\|^2
 \\ \nonumber
&&   -  \left(\frac{2}{\rho_k (1+\mu \theta_k)}-1 \right)
\left[ \left( 1-\alpha_k \right) \|x_{k+1}-x_k\|^2+\left( \alpha_k^2 - \alpha_k \right) \|x_k-x_{k-1}\|^2 \right] \\\nonumber 
&=& \|x_k-x^*\|^2 - \alpha_k \|x_{k-1}-x^*\|^2
 +  2\alpha_k \left(\alpha_k+ \frac{1-\alpha_k}{\rho_k (1+\mu \theta_k)} \right)\|x_k-x_{k-1}\|^2
\\ \nonumber
&&   -  \left( 1-\alpha_k \right) \left(\frac{2}{\rho_k (1+\mu \theta_k)}-1 \right) \|x_{k+1}-x_k\|^2,
\end{eqnarray}
which is nothing else than \eqref{Descent1}. 
\end{proof}

In order to further proceed with the convergence analysis, we have to choose the sequences $(\alpha_k)_{k \geq 1}$ and $(\rho_k)_{k \geq 1}$ such that 
$$ \lim\inf_{k \to +\infty}\delta_k >0.$$  
This is a manageable task, since we can choose for example the two sequences such that
$
\lim_{k \to +\infty} \alpha_k = \alpha \geq 0,
$
and
$
 \lim_{k \to +\infty} \rho_k = \rho > 0.
$
Recalling that $\lim_{k \to +\infty} \theta_k = 1$, we obtain 
\begin{eqnarray*} 
	\lim_{k \to +\infty} \delta_k = \left( 1-\alpha \right) \left(\frac{2}{\rho (1+\mu)}-1 \right) 
	-  2\alpha \left(\alpha+ \frac{1-\alpha}{\rho (1+\mu )}  \right)	
	=\frac{2(1-\alpha)^2}{\rho (1+\mu)} -1+\alpha-2\alpha^2,
\end{eqnarray*}
thus, in order to guarantee $\lim_{k \to +\infty} \delta_k >0$ it is sufficient to choose  $\rho$ such that 
\begin{equation} \label{RhoIneq}
	0<\rho < \frac{2}{(1 + \mu)} \frac{(1-\alpha)^2}{(2\alpha^2- \alpha+1)}.
\end{equation}

\begin{Remark}\label{Rem:Rho}{\bf (inertia versus relaxation)}
Inequality~\eqref{RhoIneq} represents the necessary trade-off between inertia and relaxation (see Figure~\ref{FracProgBox1} for two particular choices of $ \mu $). The expression is similar to the one obtained in~\cite[Remark 2.13]{AC_Prox}, the exception being an additional factor incorporating the stepsize parameter $ \mu $.
	This means  that for given $ 0 \leq \alpha < 1 $ the upper bound for the relaxation parameter is $ \overline{\rho}(\alpha, \mu) =  \frac{2}{(1+\mu)} \frac{(1-\alpha)^2}{(2\alpha^2- \alpha+1)} $.
	We further see that $ \alpha \mapsto \overline{\rho}(\alpha, \mu) $ is a decreasing function on the interval $ \left[ 0, 1 \right] $.
	Hence, the maximal value for the limit of the sequence of relaxation parameters is obtained when $ \alpha = 0 $  and is $ \rho_{\max}(\mu) := \overline{\rho}(0, \mu) = \frac{2}{1 + \mu} $.
	On the other hand, when $ \alpha \nearrow 1 $, then $ \overline{\rho}(\alpha, \mu) \searrow 0 $.
	In addition, the function $ \mu \mapsto \rho_{\max}(\mu) $ is also decreasing on $ \left[ 0, 1 \right] $ with limiting values 2 as $ \mu \searrow 0 $, and 1 as $ \mu \nearrow 1 $.
\end{Remark}
\begin{figure}[ht!]
		\centering
		\includegraphics[width=0.49\textwidth]{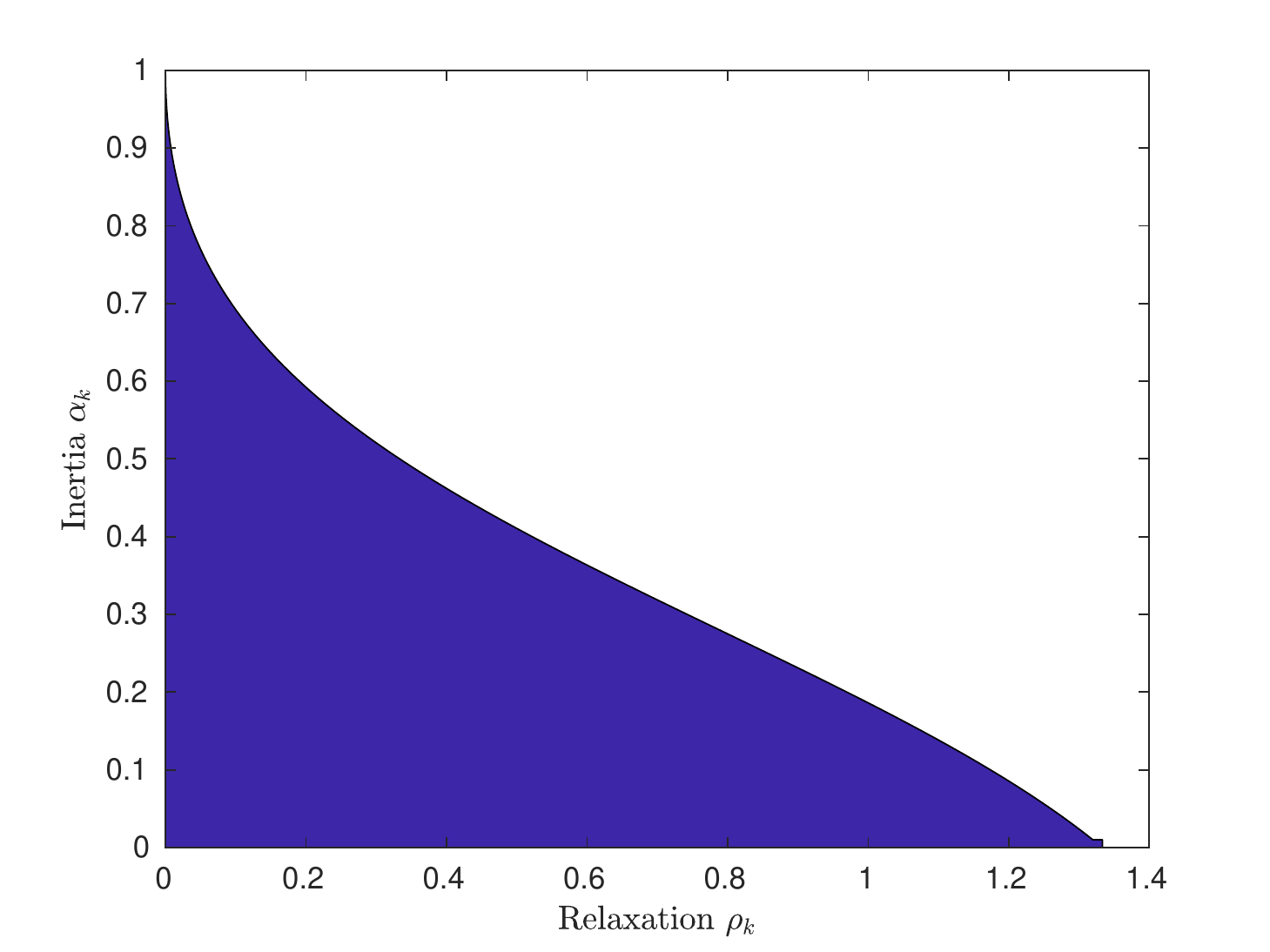}
		\includegraphics[width=0.49\textwidth]{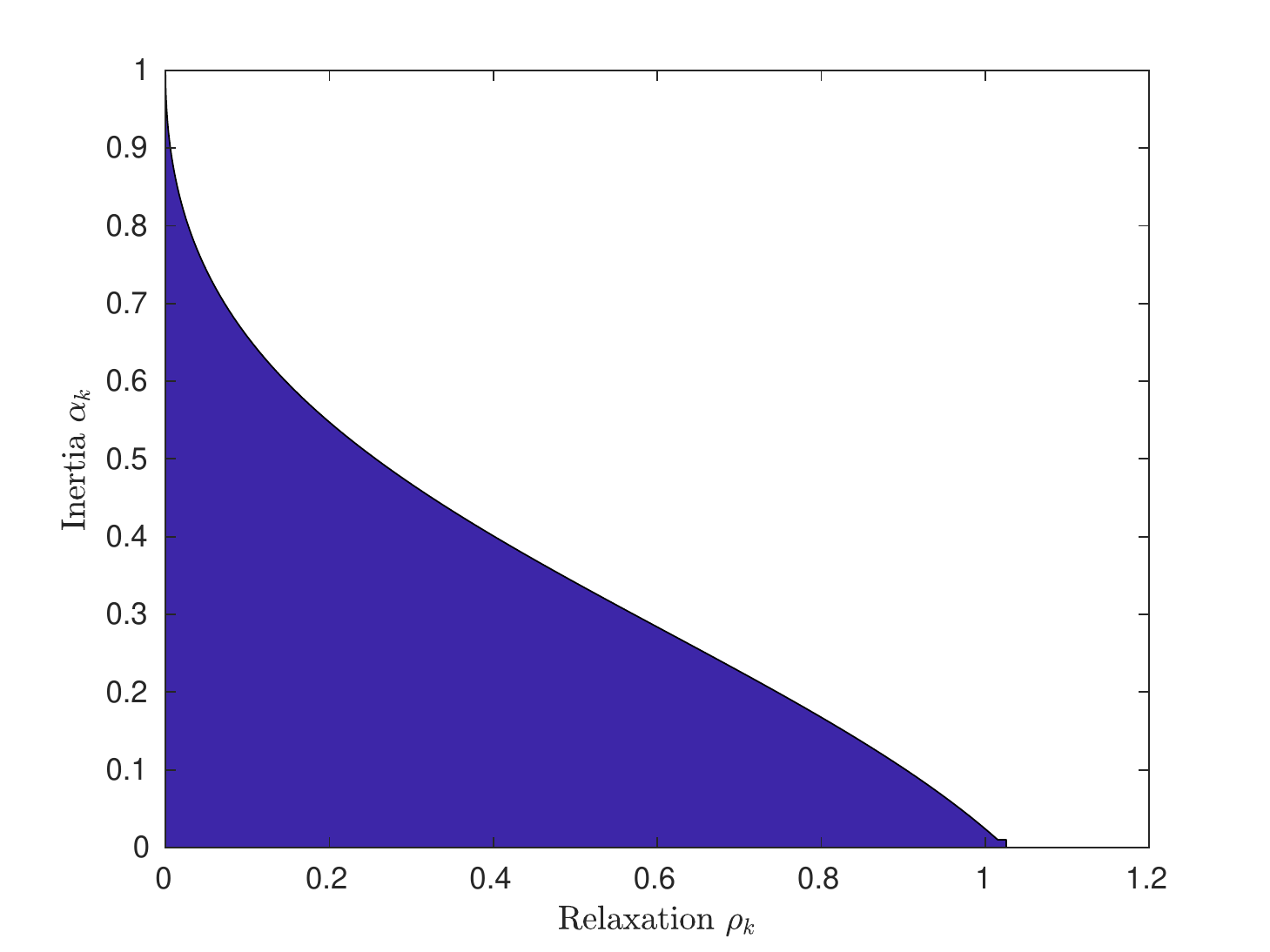}
		\caption{Trade-off between inertia and relaxation for $\mu=0.5$ (left) and $\mu=0.95$ (right).}\label{FracProgBox1}
	\end{figure}

\begin{Proposition} \label{Lem3bis}
Let $(\alpha_k)_{k \geq 1}$ be a non-decreasing sequence of non-negative numbers with $0\le \alpha_k \le \alpha <1$ for all $k\ge 1$ and $(\rho_k)_{k \geq 1}$ a sequence of positive numbers such that $ \lim\inf_{k \to +\infty}\delta_k >0$. For $x^* \in \Zeros(A+B)$ we define the sequence $(H_k)_{k \geq 1}$ as in \eqref{Hk}. Then the following statements are true:
\begin{itemize}
	\item[(i)] The sequence $(x_k)_{k \geq 0}$ is bounded.
	\item[(ii)] There exists $\lim_{k \to +\infty} H_k \in \R$.  
	\item[(iii)] $\sum_{k=1}^{+\infty} \delta_{k} \|x_{k+1}-x_{k}\|^2  < +\infty$.
\end{itemize}
\end{Proposition}
\begin{proof}
(i) From \eqref{Descent1} and $ \lim\inf_{k \to +\infty }\delta_k >0 $ we can conclude that there exists $ k_{0} \geq 1 $ such that the sequence $(H_k)_{k \geq k_0}$ is non-increasing. Therefore for all $k \ge k_0$ we have 
	$$
	H_{k_0} \geq H_k \geq \|x_{k}-x^*\|^2 -\alpha_k \|x_{k-1}-x^* \|^2 \geq \|x_{k}-x^*\|^2 -\alpha \|x_{k-1}-x^* \|^2$$
and, from here, 
	\begin{eqnarray*}
		 \|x_{k}-x^*\|^2 
		 &\leq & \alpha \|x_{k-1}-x^*\|^2 + H_{k_0}\\
		 &\leq & \alpha ^{k-k_0}  \|x_{k_0}-x^*\|^2 + H_{k_0}(1+\alpha +...+\alpha^{k-k_0-1})\\
		 &=& \alpha ^{k-k_0}  \|x_{k_0}-x^*\|^2 + H_{k_0}\frac{1-\alpha^{k-k_0}}{1-\alpha}, 		 
	\end{eqnarray*}
which implies that $(x_k)_{k \geq 0}$ is bounded and so is $(H_k)_{k \geq 1}$. 

(ii) Since  $(H_k)_{k \geq k_0}$ is non-increasing and bounded, it converges to a real number.

(iii) Follows from (ii) and Proposition \ref{Lem2bis}. 
\end{proof}

We are now in the position to prove the main result of this section. In order to do so, we first recall two useful lemmas.

\begin{Lemma} {\rm(\cite{AA01})}\label{AA} 
	Let $(\varphi_k)_{k \geq 0}$,  $(\alpha_k)_{k \geq 1}$,  and $(\psi_k)_{k \geq 1}$ be sequences of nonnegative real numbers satisfying  
	\begin{equation*}
	\varphi_{k+1}\le \varphi_k+ \alpha_k (\varphi_k-\varphi_{k-1}) +\psi_k \quad \forall k \geq 1, \quad  \sum_{k=1}^{+\infty} \psi_k < +\infty,
	\end{equation*}
and such that $0 \leq \alpha_k \leq \alpha < 1$ for all $k \geq 1$. Then the limit $\lim_{k\to +\infty} \varphi_k \in \R$ exists. 	
\end{Lemma}

\begin{Lemma}\label{Opial}{\bf (discrete Opial Lemma, \cite{Opial})}
	Let $C$ be a nonempty subset of $H$ and $(x_k)_{k \geq 0}$ be a sequence in $H$ such that the following two conditions hold:
	\begin{itemize}
		\item[(i)] For every $x\in C$, $\lim_{k\to+\infty}\|x_k-x\|$ exists.
		\item[(ii)] Every weak sequential cluster point of $(x_k)_{k \geq 0}$ is in $C$.
	\end{itemize}
	Then $(x_k)_{k \geq 0}$ converges weakly to an element in $C$.
\end{Lemma}

\begin{Theorem}\label{TheoremMain1}
Let $(\alpha_k)_{k \geq 1}$ be a non-decreasing sequence of non-negative numbers with $0\le \alpha_k \le \alpha <1$ for all $k\ge 1$ and $(\rho_k)_{k \geq 1}$ a sequence of positive numbers such that $ \lim\inf_{k \to +\infty }\delta_k >0$. Then the sequence $(x_k)_{k \geq 0}$ converges weakly to some element in $\Zeros(A+B)$ as $k \rightarrow +\infty$.
\end{Theorem}
\begin{proof}
The result will be a consequence of the discrete Opial Lemma. To this end we will prove that the conditions (i) and (ii) in Lemma \ref{Opial} for $C:=\Zeros(A+B)$ are satisfied. 

Let $x^* \in \Zeros(A+B)$. Indeed, it follows from \eqref{VV3} and \eqref{F22} that for $k$ large enough
$$
\|x_{k+1}-x^*\|^2 
\leq  (1+  \alpha_k) \|x_k-x^*\|^2 -  \alpha_k \|x_{k-1}-x^*\|^2  +\alpha_k(1+\alpha_k) \|x_k - x_{k-1}\|^2.
$$
Therefore, according to Lemma \ref{AA} and Proposition \ref{Lem3bis} (iii), $\lim_{k \to +\infty} \|x_k -x^*\|$ exists.

Let $\bar{x}$ be a weak limit point of $(x_k)_{k \geq 0}$ and a subsequence $(x_{k_l})_{l \geq 0}$ which converges weakly to $\bar{x}$ as $l \to +\infty$. 
From Proposition \ref{Lem3bis} (iii)  we have 
	$$
	 \lim_{k \to +\infty } \delta_k \|x_{k+1}-x_{k}\|^2 =0,
	 $$
which, as $ \lim\inf_{k \to +\infty }\delta_k >0$, yields $\lim_{k \to +\infty } \|x_{k+1}-x_{k}\| =0 $.
	 Since 
	 $$
	 \|t_k -z_k \| =\frac{1}{\rho_k} \|x_{k+1} -z_k\|= \frac{1}{\rho_k} \|x_{k+1} -x_k + \alpha_k  (x_k -x_{k-1})\| \quad \forall k \geq 1,
	 $$ 
we have $\lim_{k \to +\infty} \|t_{k}-z_{k}\| = 0.$
On the other hand, for all $k \geq 1$ holds
\begin{eqnarray*} 
	\|t_k-z_k\|	&=&  \|y_k-z_k + \lambda_k (By_k-Bz_k)\| \ge  \|y_k-z_k\| - \lambda_k\|By_k-Bz_k\| \\
	&\geq &  \left( 1-\mu \theta_k\right) \|y_k-z_k\|. 
\end{eqnarray*}
	 Since $\lim_{k \to +\infty} \left( 1-\mu \theta_k\right) = 1-\mu > 0$, we can conclude that
	 $
	 \|y_{k} - z_k\| \to 0 
	 $
	  as $k \to +\infty$. This shows that $(y_{k_l})_{l \geq 0}$  and $(z_{k_l})_{l \geq 0}$ converge weakly to $\bar{x}$ as $l \to +\infty$. The definition of $(y_{k_l})_{l \geq 0}$ gives
$$
	 \frac{1}{\lambda_{k_l}} (z_{k_l} -y_{k_l}) + By_{k_l} - Bz_{k_l} \in (A+B)y_{k_l} \quad \forall l \geq 0. 
	 $$
Using that $\left( \frac{1}{\lambda_{k_l}} (z_{k_l} -y_{k_l}) + By_{k_l} - Bz_{k_l}\right)_{l \geq 0} $ converges strongly to $0$ and the graph of the maximal monotone operator $A+B$ is sequentially closed with respect to the weak-strong topology of the product space $H\times H $, we obtain $0 \in (A+B)\bar x$, thus $\bar x \in \Zeros(A+B)$.
\end{proof}

\begin{Remark}\label{pseudomonotone}
In the particular case of the variational inequality \eqref{VIP}, which corresponds to the case when $A$ is the normal cone of a nonempty closed convex subset $C$ of $H$, by taking into account that $J_{\lambda N_C} = P_C$ is for all $\lambda >0$ the projection operator onto $C$, the relaxed inertial FBF algorithm reads
\begin{equation*}
	(RIFBF-VI) \quad \quad \quad (\forall k \geq 1) \
	\begin{cases}
	z_k=x_k+\alpha_k(x_k-x_{k-1})\\
	y_k=P_C (I-\lambda_k B) z_k\\
	x_{k+1}=(1-\rho_k) z_k +\rho_k \left(y_k-\lambda_k(By_k-Bz_k) \right),	
	\end{cases}
\end{equation*}
where $x_0,x_1 \in H$ are starting points, $(\lambda_{k})_{k \geq 1}$ and $(\rho_{k})_{k \geq 1}$ are sequences of positive numbers, and $(\alpha_{k})_{k \geq 1}$ is a sequence of non-negative numbers.

The algorithm converges weakly to a solution of \eqref{VIP} when $B$ is a monotone and Lipschitz continuous operator in the hypotheses of Theorem \ref{TheoremMain1}.  However, we want to point out that it converges even if $B$ is pseudo-monotone on $H$, Lipschitz-continuous and  sequentially weak-to-weak continuous, respectively, if $H$ is finite dimensional, and $B$ is pseudo-monotone on $C$ and Lipschitz continuous.

We recall that $B$ is said to be pseudo-monotone on $C$ (on $H$) if for all $x,y \in C \ (x,y \in H)$ it holds
$$
\langle Bx,y-x\rangle\geq 0\;\Rightarrow\; \langle By,y-x\rangle\geq 0. 
$$
Denoting $t_k:=y_k -  \lambda_k (By_k-Bz_k)$ and $\theta_k:=\frac{\lambda_k}{\lambda_{k+1}}$ for all $k \geq 1$, then for all $x^* \in \Zeros(N_C + B)$  it holds
\begin{equation*}
	\|t_k-x^*\|^2\le \|z_k-x^*\|^2- \left(1- \mu^2 \theta_k^2 \right) \|y_{k} - z_k\|^2 \quad \forall k \geq 1
	\end{equation*}
which is nothing else than relation \eqref{MainIne}.

Indeed, since $y_k \in C$, we have  $\left\langle Bx^*, y_k-x^*\right\rangle \geq 0,$  and, further, by the pseudo-monotonicity of $B$, it holds $\left\langle By_k, y_k-x^*\right\rangle \geq 0$ for all $k \geq 1$. On the other hand, since $y_k=P_C(I-\lambda_k B) z_k$,  we have $\left\langle x^*-y_k,  y_k - z_k+\lambda_k Bz_k \right\rangle  \geq 0$ for all $k \geq 1$. The two inequalities yield
$$
\left\langle y_k - x^*,  z_k - t_k \right\rangle \geq 0 \quad \forall k \geq 1,$$
which, combined with \eqref{intermed}, lead as in the proof of Proposition \ref{Prop1} to the conclusion.

Now, since \eqref{MainIne} holds, the statements in Proposition \ref{Lem2bis} and Proposition \ref{Lem3bis} remain true and, as seen in the proof of Theorem \ref{TheoremMain1}, they guarantee that the limit $\lim_{k \rightarrow +\infty} \|x_k-x^*\| \in \R$ exists, and that $\lim_{k \rightarrow +\infty} \|y_k-z_k\| = \lim_{k \rightarrow +\infty} \|By_k-Bz_k\| = 0$. Having that, the weak converge of $(x_k)_{k \geq 0}$ to a solution of \eqref{VIP} follows, by arguing as in the proof of \cite[Theorem 3.1]{BCV18}, when $B$ is pseudo-monotone on $H$, Lipschitz-continuous and  sequentially weak-to-weak continuous, and as in the proof of \cite[Theorem 3.2]{BCV18}, when $H$ is finite dimensional, and $B$ is pseudo-monotone on $C$ and Lipschitz continuous.
\end{Remark}

\section{Numerical experiments} 
\label{sec:Numerical}

In this section, we provide two numerical experiments that complete our theoretical results. The first one, a bilinear saddle point problem, is a usual deterministic example where all the assumptions are fulfilled to guarantee convergence. For the second experiment we leave the save harbour of justified assumptions and frankly use (RIFBF) to treat a more complex (stochastic) problem and train a special kind of generative machine learning system that receives a lot of attention recently.

\subsection{Bilinear saddle point problem}
\label{ssec:Bilinear}

We want to solve the saddle-point problem
\begin{equation*}
	\min_{\theta \in \Theta} \max_{\varphi \in \Phi} \, V(\theta, \varphi),
\end{equation*}
in the sense that we want to find
$ \theta^{\ast} \in \Theta $ and $ \varphi^{\ast} \in \Phi $ such that
\begin{equation*}
	V(\theta^{\ast}, \varphi) \leq V(\theta^{\ast}, \varphi^{\ast}) \leq V(\theta, \varphi^{\ast})
\end{equation*}
for all $ \theta \in \Theta $ and all $ \varphi \in \Phi $, where $ \Theta \subseteq \mathbb{R}^{m} $ and $ \Phi \subseteq \mathbb{R}^{n} $  are nonempty, closed and convex sets and $V(\theta, \varphi) = \theta^{T} A \varphi + a^{T} \theta + b^{T} \varphi $ with $ A \in \mathbb{R}^{m \times n} $, $ a \in \mathbb{R}^{m} $ and $ b \in \mathbb{R}^{n} $.
The  monotone inclusion to be solved in this case is of the form
\begin{equation*}
	0 \in N_{\Theta \times \Phi}(\theta, \varphi) + F(\theta, \varphi),
\end{equation*}
where
$ F(\theta, \varphi) =
M
\left(
	\begin{array}{c}
		\theta\\
		\varphi
	\end{array}
\right)
+
\left(
	\begin{array}{c}
		 a\\
		-b
	\end{array}
\right)$,
with 
$ M =
\left( 
	\begin{array}{cc}
		0		& A\\
		-A^{T}	& 0
	\end{array} 
\right),$
is monotone and Lipschitz continuous with Lipschitz constant $ L = {\Vert M \Vert}_{2}$. Notice, that $ F $ is not cocoercive.

In our experiment we chose $ m = n = 500 $, and $ A $, $ a $ and $ b $ to have random entries drawn from a uniform distribution on the interval $ \left[0, 1\right] $.
The constraint sets $ \Theta $ and $ \Phi $ were chosen to be unit balls in the Euclidean norm.
Furthermore we took constant stepsize $ \lambda_{k} = \lambda = \frac{\mu}{L} $ for all $ k \geq 1$, where $ 0 < \mu < 1 $, constant inertial parameter $ \alpha_{k} = \alpha $ for all $ k \geq 1$ and constant relaxation parameter $ \rho_{k} = \rho $ for all $ k \geq 1 $.
We set $x_{k} := (\theta_{k}, \varphi_{k}) $ to fit the framework of our algorithm.
The starting point $ x_{0} $ is initialised randomly (entries drawn from the uniform distribution on $ \left[0, 1\right] $) and we set $ x_{1} = x_{0} $.
In Table~\ref{tab:mu-05} we can see the necessary number of iterations for the algorithm to reach $ \Vert y_{k} - z_{k} \Vert \leq \varepsilon $ for different choices of the parameters $ \alpha $ and $ \rho $ and $ \mu = 0.5 $.
The maximum number of iterations was set to be $ 10^{4} $, hence entries with ``1000'' actually mean
``$ \geq 10000 $''.
Since we did not observe different behaviour for various random trials, we provide the results for only one run in the following.

As mentioned in Remark~\ref{Rem:Rho}, there is a trade-off between inertia and relaxation. The parameters $ \alpha $ and $ \rho $ also need to fulfil the relations
\begin{equation*}
	0 \leq \alpha < 1
	\hspace{3mm} \text{and} \hspace{3mm}
	0 < \rho < \frac{2}{(1 + \mu)} \frac{(1 - \alpha)^{2}}{(2 \alpha^{2} - \alpha + 1)},
\end{equation*}
which is the reason why not every combination of $ \alpha $ and $ \rho $ is valid.

\begin{table}[h!]
	\centering
	\includegraphics[width=0.95\textwidth]{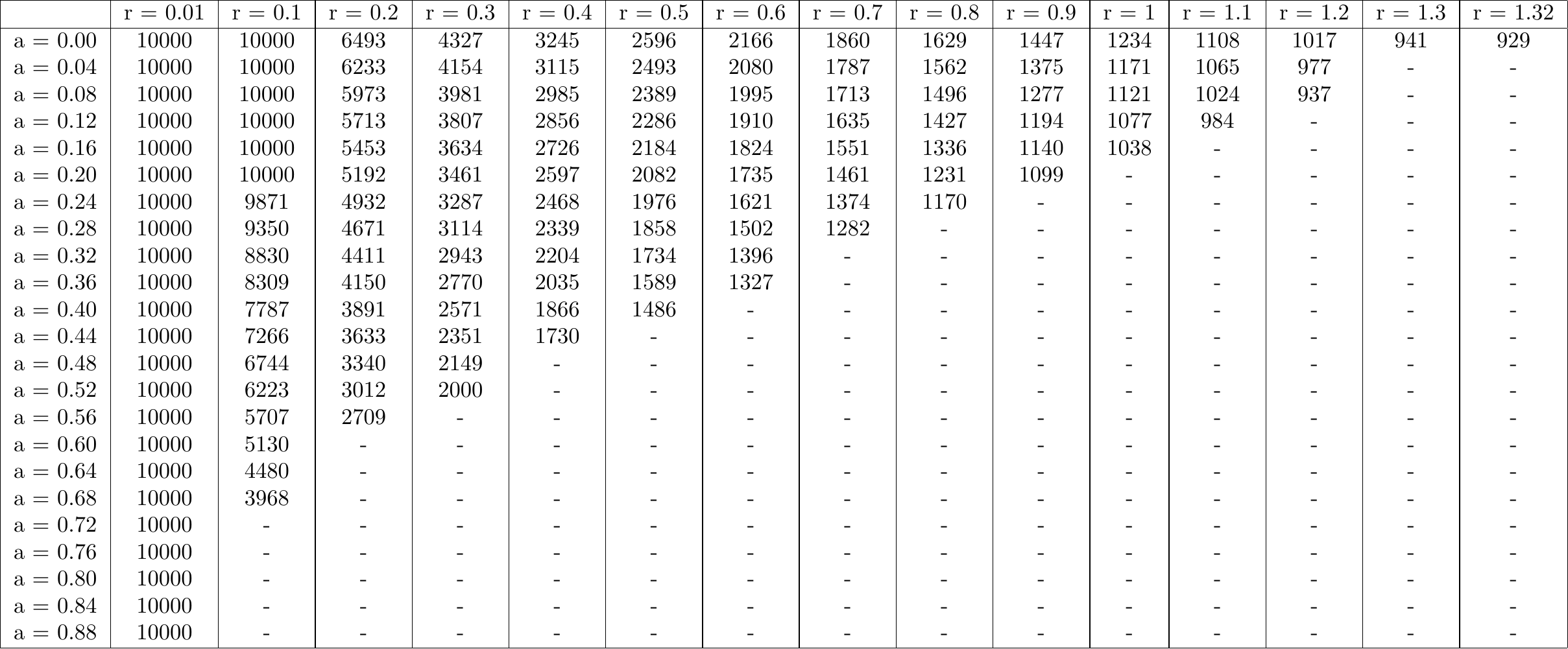}
	\caption{Number of iterations when $\mu=0.5$ and $\varepsilon=10^{-5}$ with constant relaxation parameter~(r) and constant inertial parameter~(a) (maximum number of iterations was $ 10^{4} $).}\label{tab:mu-05}
\end{table}

We see that for a particular choice of the relaxation parameter the least number of iterations is achieved when the inertial parameter is as large as possible.
If we fix the inertial parameter we observe that also larger values of $ \rho $ are better and lead to fewer iterations.
To get a conclusion regarding the trade-off between the two parameters, Table~\ref{tab:mu-05} suggests that the influence of the relaxation parameter is stronger than that of the inertial parameter.
Even though the possible values for $ \alpha $ get smaller if $ \rho $ goes to $ \rho_{\max}(\mu) = \frac{2}{1+\mu} $, the number of iterations gets less for $ \alpha > 0 $. In particular, we want to point out that over-relaxation ($ \rho > 1 $) seems to be highly beneficial.

Comparing the results for different choices of $ \mu $ in Table~\ref{tab:mu-09} ($ \mu = 0.9 $) and Table~\ref{tab:mu-01} ($ \mu = 0.1 $), we can observe very interesting behaviour.
Even though smaller $ \mu $ allows for larger values of the relaxation parameter (see Remark~\ref{Rem:Rho}), larger $ \mu $ leads to better results in general.
This behaviour presumably is due to the role $ \mu $ plays in the definition of the stepsize.

\begin{figure}[h!]
	\centering
	\includegraphics[width=0.7\textwidth]{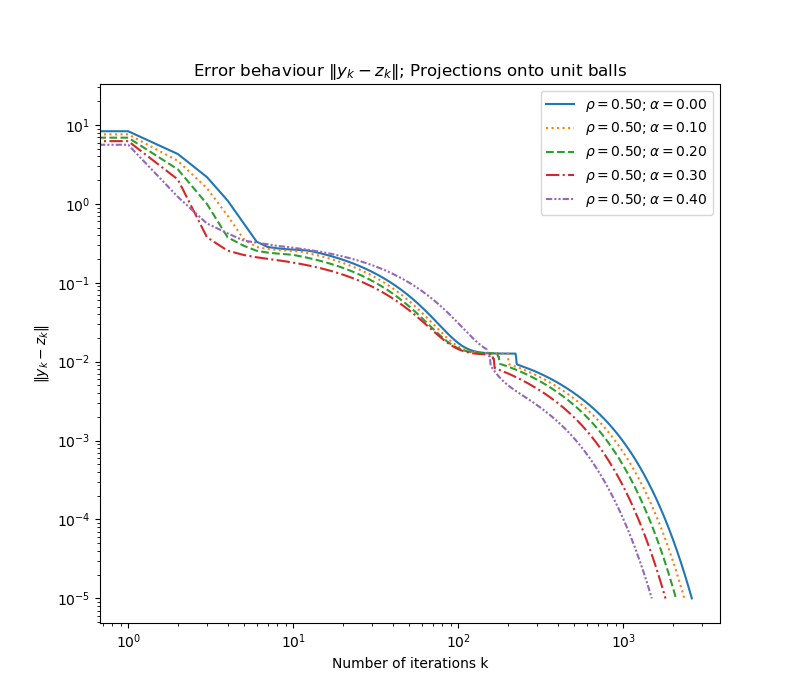}
	\caption{Behaviour of $ \Vert y_{k} - z_{k} \Vert $ for $\mu=0.5$ and $\varepsilon=10^{-5}$ with constant relaxation parameter~($ \rho $) and constant inertial parameter~($ \alpha $)}\label{fig:err_fp}
\end{figure} 

In Figure~\ref{fig:err_fp} we see the development of $ \Vert y_{k} - z_{k} \Vert $ for $ \rho = 0.5 $ and various values for $ \alpha $.
We see that the behaviour of the residual is similar for most combinations of parameters, i.e., for larger values of the inertial parameter the behaviour is better in general throughout the whole run.
However when $ \alpha $ gets close to the limiting case the behaviour of the residual is not consistently better anymore.
Temporarily the residual is even worse than for smaller $ \alpha $, nevertheless the algorithm still terminates in fewer iterations in the end.

\begin{table}[h!]
	\centering
	\includegraphics[width=0.95\textwidth]{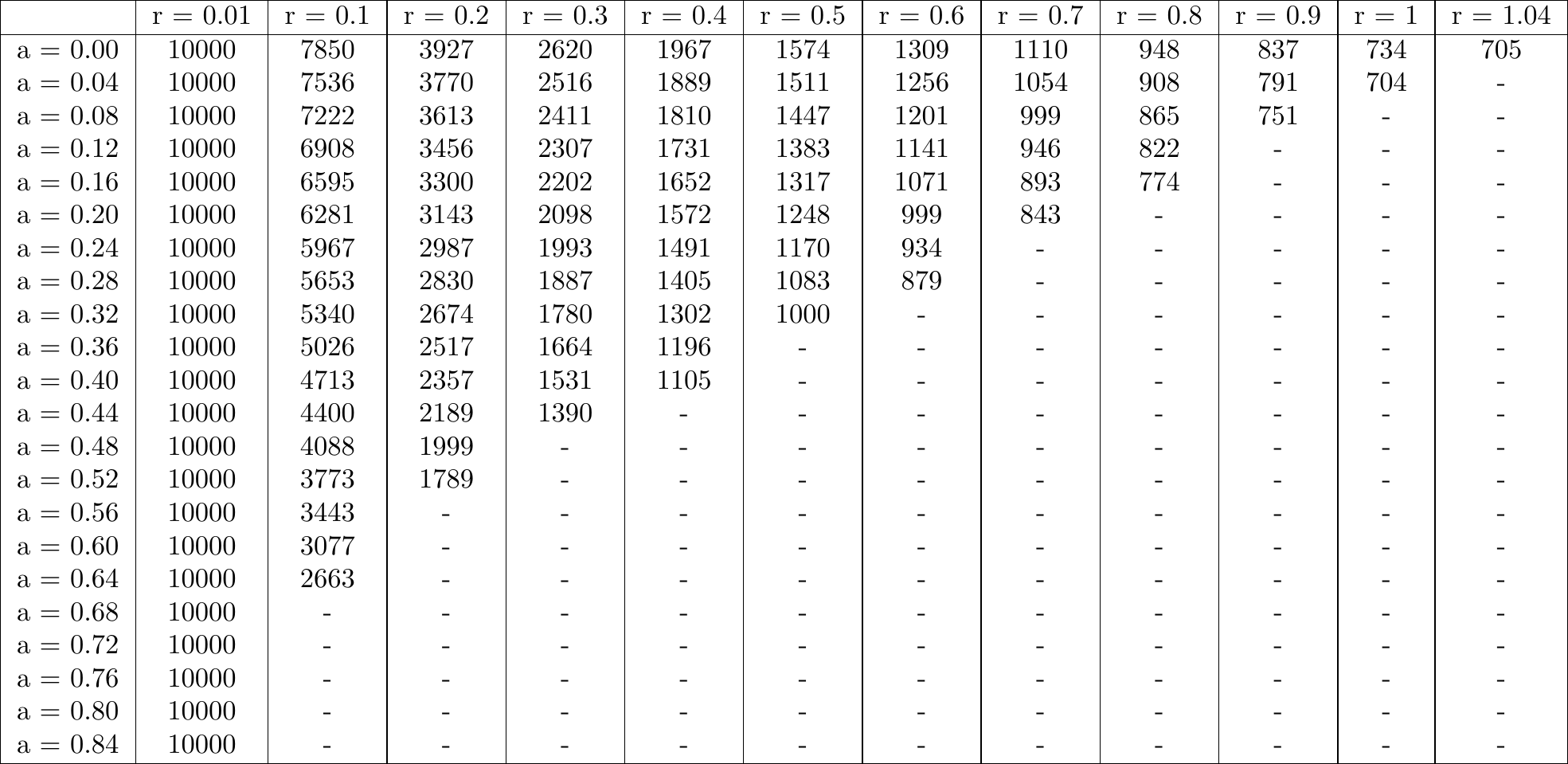}
	\caption{Number of iterations when $\mu=0.9$ and $\varepsilon=10^{-5}$ with constant relaxation parameter~(r) and constant inertial parameter~(a) (maximum number of iterations was $ 10^{4} $).}\label{tab:mu-09}
\end{table}

To get further insight into the convergence behaviour, we also look at the following gap function,
\begin{equation*}
	G(s, t) := \inf_{\theta \in \Theta} V(\theta, t) - \sup_{\varphi \in \Phi} V(s, \varphi).
\end{equation*}
The quantity $ (G(\theta_{k}, \varphi_{k}))_{k \geq 0} $ should be a measure to judge the performance of the iterates $ (\theta_{k}, \varphi_{k})_{k \geq 0} $, as for the optimum $ (\theta^{\ast}, \varphi^{\ast}) $ we have $ V(\theta^{\ast}, \varphi) \leq V(\theta^{\ast}, \varphi^{\ast}) \leq V(\theta, \varphi^{\ast}) $ for all $ \theta \in \Theta $ and all $ \varphi \in \Phi $ and hence $ G(\theta^{\ast}, \varphi^{\ast}) = 0 $.\\
Because of the particular choice of a bilinear objective and the constraint sets $ \Theta $ and $ \Phi $ the expressions can be actually computed in closed form and we get
\begin{equation*}
	G(\theta_{k}, \varphi_{k}) = - \Vert A \varphi_{k} + a \Vert_{2} + b^{T}\varphi_{k} - \Vert A^{T}\theta_{k} + b \Vert_{2} - a^{T}\theta_{k}  \quad \forall k \geq 0.
\end{equation*}

In Figure~\ref{fig:err_gap} we see the development of the absolute value of the gap, $ \lvert G(\theta_{k}, \varphi_{k}) \rvert $ for $ \rho = 0.5 $ and various values for $ \alpha $.
We see that, as in the case for the residual of the fixed point iteration, the behaviour is similar for most combinations of parameters, i.e., for larger values of the inertial parameter the behaviour is better in general throughout the whole run.
However when $ \alpha $ gets close to the limiting case the behaviour of the gap is not consistently better anymore.
As the theory suggests, the gap indeed decreases and tends to zero as the number of iterations grows larger.

\begin{figure}[h!]
	\centering
	\includegraphics[width=0.7\textwidth]{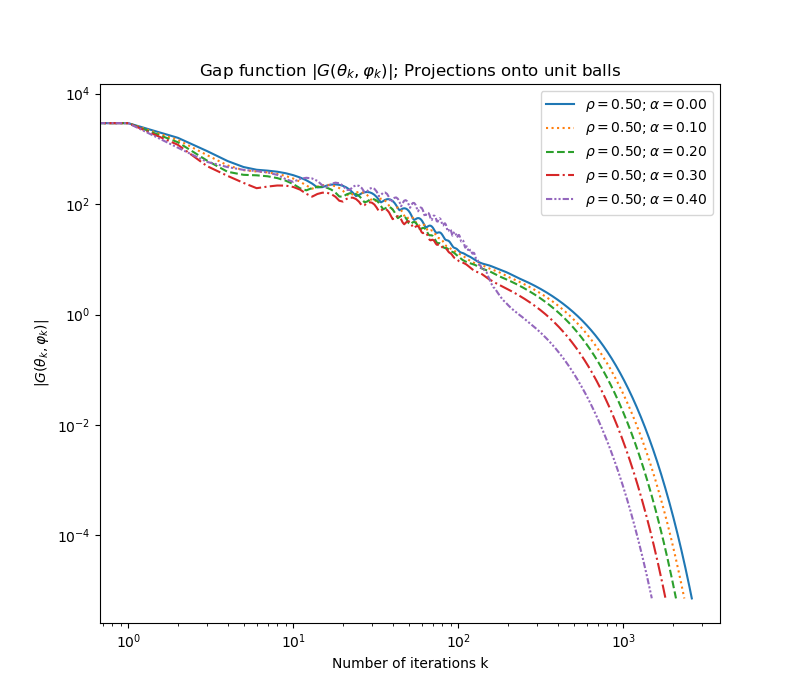}
	\caption{Behaviour of $ \lvert G(\theta_{k}, \varphi_{k}) \rvert $ for $\mu=0.5$ and $\varepsilon=10^{-5}$ with constant relaxation parameter~($ \rho $) and constant inertial parameter~($ \alpha $).}\label{fig:err_gap}
\end{figure}

\begin{table}
	\centering
	\includegraphics[angle=90,origin=c,height=0.95\textwidth]{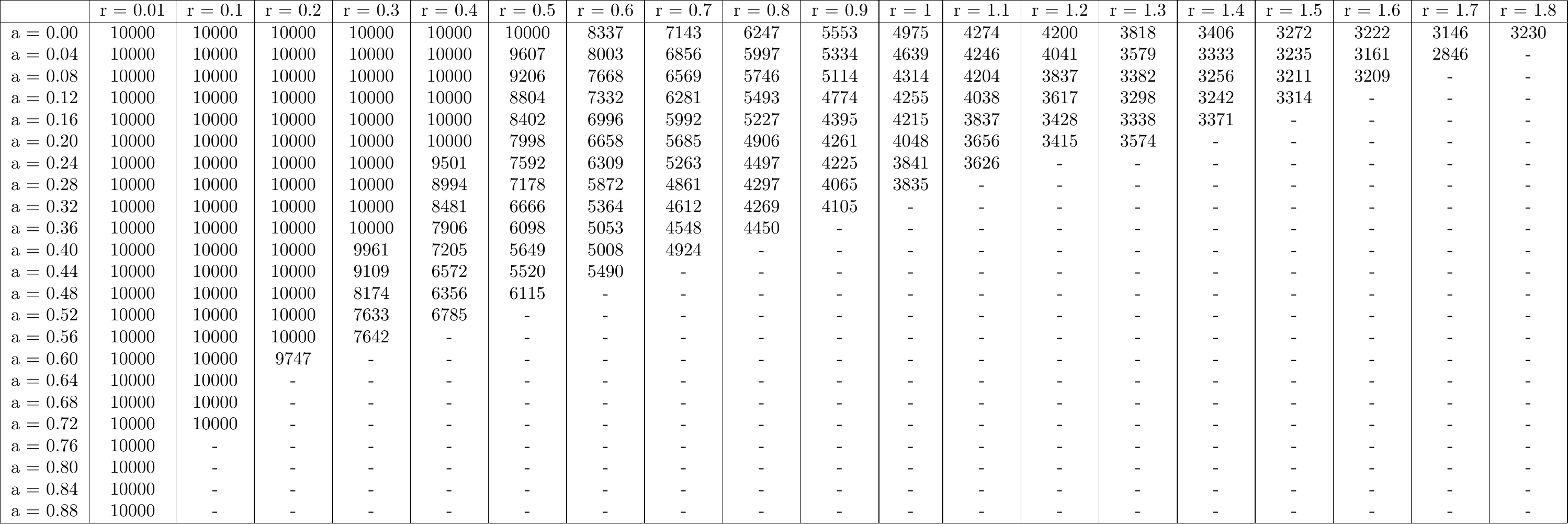}
	\caption{Number of iterations when $\mu=0.1$ and $\varepsilon=10^{-5}$ with constant relaxation parameter~(r) and constant inertial parameter~(a) (maximum number of iterations was $ 10^{4} $).}\label{tab:mu-01}
\end{table}

\subsection{Generative Adversarial Networks (GANs)}
\label{ssec:GAN}

Generative Adversarial (Artificial Neural) Networks (GANs) is a class of machine learning systems, where two ``adversarial'' networks compete in a (zero-sum) game against each other. Given a training set, this technique aims to learn to generate new data with the same statistics as the training set.
The generative network, called generator, tries to mimic the original genuine data (distribution) and strives to produce samples that fool the other network. The opposing, discriminative network, called discriminator, evaluates both true samples as well as generated ones and tries to distinguish between them.
The generator's objective is to increase the error rate of the discriminative network by producing novel samples that the discriminator thinks were not generated, while its opponent's goal is to successfully judge (with high certainty) whether the presented data is true or not.
Typically, the generator learns to map from a latent space to a data distribution which is (hopefully) similar to the original distribution. However, one does not have access to the distribution itself, but only random samples drawn from it.
\begin{figure}
	\centering
	\includegraphics[width=0.8\textwidth]{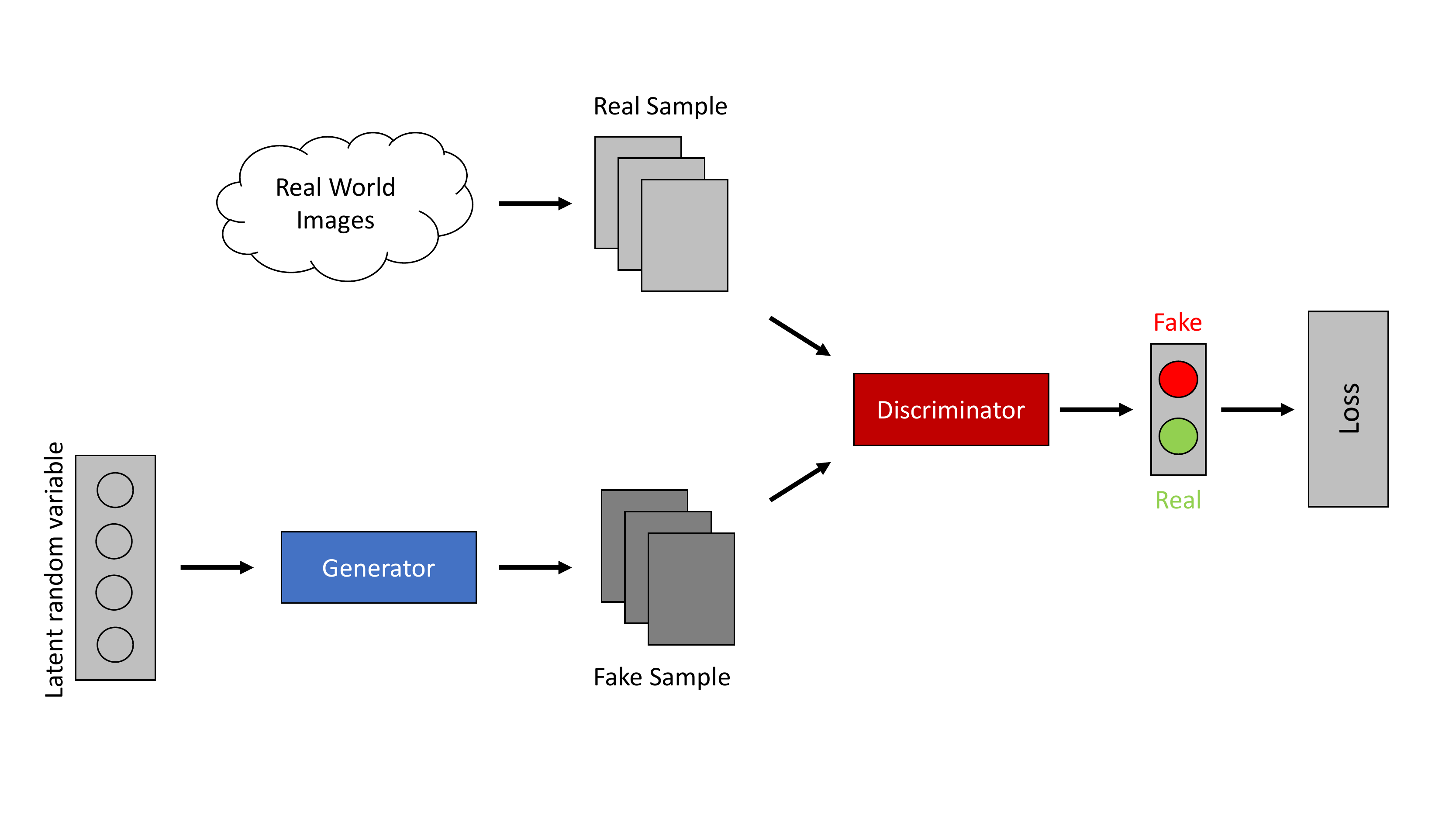}
	\caption{Schematic illustration of Generative Adversarial Networks.}
\end{figure}
The original formulation of GANs by Goodfellow~et~al.~\cite{Goodfellow} is as follows,
\begin{equation*}
	\min_{\theta} \max_{\varphi} V(\theta, \varphi),
\end{equation*}
where $ V(\theta, \varphi) = \mathbb{E}_{x \sim p} \left[ \log \left( D_{\varphi} (x) \right) \right] + \mathbb{E}_{x' \sim q_{\theta}} \left[ \log \left( 1 - D_{\varphi} (x') \right) \right]$ is the value function, $ \theta $ and $ \varphi $ is the parametrisation of the generator and discriminator, respectively, $ D_{\varphi} $ is the probability how certain the discriminator is that the input is real, and $ p $ and $ q_{\theta} $ is the real and learned distribution, respectively.

Again, the problem
\begin{equation*}
	\min_{\theta \in \Theta} \max_{\varphi \in \Phi} \, V(\theta, \varphi)
\end{equation*}
is understood in the sense that we want to find $ \theta^{\ast} \in \Theta $ and $ \varphi^{\ast} \in \Phi $ such that
\begin{equation*}
	V(\theta^{\ast}, \varphi) \leq V(\theta^{\ast}, \varphi^{\ast}) \leq V(\theta, \varphi^{\ast})
\end{equation*}
for all $ \theta \in \Theta $ and all $ \varphi \in \Phi $. The corresponding inclusion problem then reads
\begin{equation*}
	0 \in N_{\Theta \times \Phi}(\theta, \varphi) + F(\theta, \varphi),
\end{equation*}
where $ N_{\Theta \times \Phi} $ is the normal cone operator to $ \Theta \times \Phi $ and $ F(\theta, \varphi) =  \left( \nabla_{\theta}V(\theta, \varphi), -\nabla_{\varphi}V(\theta, \varphi) \right)^{T}$.\\
If $ \Theta $ and $ \Phi $ are nonempty, convex and closed sets, and $ V(\theta, \varphi) $ is convex-concave, Fr\'echet differentiable and has a Lipschitz continuous gradient, then we have a variational inequality, the obtained theory holds and we can apply (RIFBF-VI).

This motivates to use (variants of) FBF methods for the training of GANs, even though in practice the used value functions typically are not convex-concave and further, the gradient might not be Lipschitz continuous if it exists at all. Additionally, in general one needs stochastic versions of the used algorithms and which we do not provide in the case of (RIFBF).

Recently, a first successful attempt of using methods coming from the field of variational inequalities for GAN training was done by Gidel et al.~\cite{VIP-GAN}. In particular they applied the well established extra-gradient algorithm and some derived variations. In this spirit we frankly apply the FBF method and a variant with inertial effect ($ \alpha = 0.05 $), as well as a primal-dual algorithm for saddle point problems, introduced by Hamedani and Aybat~\cite{Aybat} (``PDSP''), which to the best of our knowledge has not been used for GAN training before, and compare the results to the best method (``ExtraAdam'') from~\cite{VIP-GAN}.
\begin{table}
	\begin{center}
		\begin{tabular}{c}
			\hline
			\textbf{Generator}\\
			\hline
			
			\\
			\textit{Input:} $ z \in \mathbb{R}^{128} \sim \mathcal{N}(0, I) $\\
			Linear $128  \to 512 \times 4 \times 4$\\
			Batch Normalization\\
			ReLU\\
			transposed conv. (kernel: $ 4 \times 4 $, $ 512 \to 256 $, stride: 2, pad: 1)\\
			Batch Normalization\\
			ReLU\\
			transposed conv. (kernel: $ 4 \times 4 $, $ 256 \to 128 $, stride: 2, pad: 1)\\
			Batch Normalization\\
			ReLU\\
			transposed conv. (kernel: $ 4 \times 4 $, $ 128 \to 3 $, stride: 2, pad: 1)\\
			\textit{Tanh(.)}\\
			\\
			
			\hline
			\textbf{Discriminator}\\
			\hline
			
			\\
			\textit{Input:} $ x \in \mathbb{R}^{3 \times 32 \times 32} $\\
			conv. (kernel: $ 4 \times 4 $, $ 1 \to 64 $, stride: 2, pad: 1)\\
			LeakyReLU (negative slope: 0.2)\\
			conv. (kernel: $ 4 \times 4 $, $ 64 \to 128 $, stride: 2, pad: 1)\\
			Batch Normalization\\
			LeakyReLU (negative slope: 0.2)\\
			conv. (kernel: $ 4 \times 4 $, $ 128 \to 256 $, stride: 2, pad: 1)\\
			Batch Normalization\\
			LeakyReLU (negative slope: 0.2)\\
			Linear $ 128 \times 4 \times 4 \times 4 \to 1 $\\
			\\
			
			\hline
		\end{tabular}
	\end{center}
	\caption{DCGAN architecture for our experiments on CIFAR10.}
	\label{tab:arch}
\end{table}
For our experiments we use the DCGAN architecture (see~\cite{DCGAN}) with the WGAN objective and weight clipping (see~\cite{WGAN}) to train it on the CIFAR10 dataset (see~\cite{CIFAR10}). Note that in absence of bound constraints on the weights of at least one of the two networks, the backward step in the FBF algorithm is redundant, as we would project on the whole space and we obtain the unconstrained extra-gradient method.\\
Furthermore, in our experiments instead of stochastic gradients we use the ``Adam'' optimizer with the hyperparameters ($ \beta_{1} = 0.5 $, $ \beta_{2} = 0.9 $) that where used in~\cite{VIP-GAN}, as the best results there were achieved with this choice.
Also we would like to mention that we only did a hyperparameter search for the stepsizes of the newly introduced methods, all other parameters we chose to be equal as in the aforementioned work.

\begin{table}
	\begin{center}
		\begin{tabular}{|l|c|c|}
			\hline
			Method & IS & FID \\
			\hline
			\hline
			PDSP Adam & 4.20 $\pm$ 0.04 & 53.97 $\pm$ 0.28\\
			\hline
			Extra Adam & 4.07 $\pm$ 0.05 & 56.67 $\pm$ 0.61\\
			\hline
			FBF Adam & 4.54 $\pm$ 0.04 & 45.85 $\pm$ 0.35\\
			\hline
			\textit{IFBF Adam} $\mathit{(\alpha = 0.05)}$ & $\mathit{4.59 \pm 0.04}$ & $\mathit{45.25 \pm 0.60}$\\
			\hline
		\end{tabular}
	\end{center}
	\caption{Best IS and FID scores (averaged over 5 runs) achieved on CIFAR10.}
	\label{tab:GAN}
\end{table}

The model is evaluated using the inception score (IS) (reworked implementation by~\cite{IS} that fixes some issues of the original one) as well as the Fr{\' e}chet inception distance (FID)~(see \cite{FID}), both computed on 50,000 samples. Experiments were run with 5 random seeds for 500,000 updates of the generator on a NVIDIA GeForce RTX 2080Ti GPU. Table~\ref{tab:GAN} reports the best IS and FID achieved by each considered method. Note that the values of IS for Extra Adam differ from those stated in~\cite{VIP-GAN}, due to the usage of the corrected implementation of the score.

We see that even though we only have proved convergence for the monotone case in the deterministic setting, the variants of (RIFBF) perform well in the training of GANs. IFBF Adam outperforms all other considered methods, both for the IS and the FID. As the theory suggests, making use of some inertial effects (regardless that Adam already incorporates some momentum) seems to provide additional improvement of the numerical method in practice. The results suggest, that employing methods that are designed to capture the nature of a problem, in this case a constrained minimax/saddle-point problem, is highly beneficial. In Figure~\ref{fig:samples} we provide samples of the generator trained with the different methods.

\begin{figure}
    \centering
    \begin{subfigure}[b]{0.475\textwidth}
        \centering
        \includegraphics[width=\textwidth]{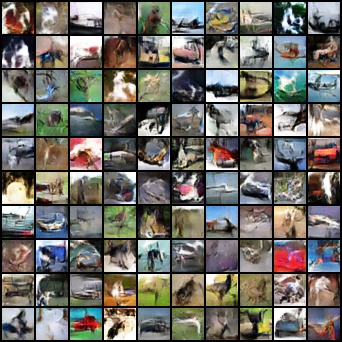}
        \caption{PDSP Adam}
        \label{fig:sample_aybat}
    \end{subfigure}
    \hfill
    \begin{subfigure}[b]{0.475\textwidth}
        \centering
        \includegraphics[width=\textwidth]{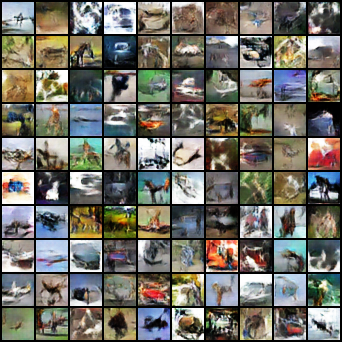}
        \caption{Extra Adam}
        \label{fig:sample_extra}
    \end{subfigure}
    \vskip\baselineskip
    \begin{subfigure}[b]{0.475\textwidth}
        \centering
        \includegraphics[width=\textwidth]{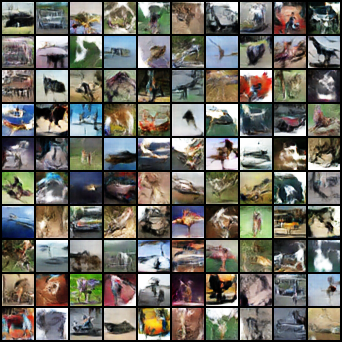}
        \caption{FBF Adam}
        \label{fig:sample_fbf}
    \end{subfigure}
    \quad
    \begin{subfigure}[b]{0.475\textwidth}
        \centering
        \includegraphics[width=\textwidth]{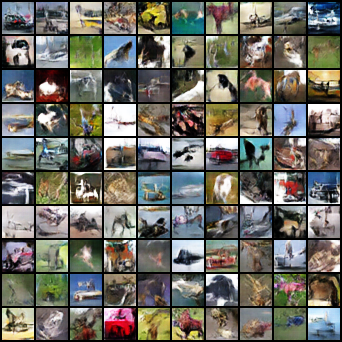}
        \caption{IFBF Adam}
        \label{fig:sample_ifbf}
    \end{subfigure}
    \caption{Comparison of the samples of a WGAN with weight clipping trained with the different methods.}
    \label{fig:samples}
\end{figure}

{\bf Acknowledgements.} The authors would like to thank Julius Berner and Axel B\"ohm for fruitful discussions and their valuable suggestions and comments.

\end{document}